\documentclass[11pt]{article}
\usepackage{a4wide}

\usepackage[a4paper, margin=2.3cm]{geometry}
\usepackage{amsmath,amssymb,amsthm}
\usepackage{color}
\usepackage{parskip}
\usepackage{hyperref}
\usepackage{parskip}
\usepackage{setspace}
\usepackage{graphicx}

\newtheorem{theorem}{Theorem}[section]
\newtheorem{remark}{Remark}[section]
\newtheorem{corollary}[theorem]{Corollary}

\newtheorem{lemma}[theorem]{Lemma}

\newtheorem{definition}[theorem]{Definition}
\newtheorem{conjecture}[theorem]{Conjecture}
\newtheorem*{definition*}{Definition}

\begin{document}
\title{Mattila--Sj\"{o}lin type functions: A finite field model}

\author{Daewoong Cheong \thanks{Department of Mathematics, Chungbuk National University. Email: {\tt daewoongc@chungbuk.ac.kr}}\and
Doowon Koh\thanks{Department of Mathematics, Chungbuk National University. Email: {\tt koh131@chungbuk.ac.kr}}
\and
    Thang Pham\thanks{Department of Mathematics,  ETHZ Switzerland. Email: {\tt phamanhthang.vnu@gmail.com}}
  \and
  Chun-Yen Shen \thanks{Department of Mathematics,  National Taiwan University. Email: {\tt cyshen@math.ntu.edu.tw}}
}
\date{}
\maketitle

\begin{abstract}
Let $\phi(x, y)\colon \mathbb{R}^d\times \mathbb{R}^d\to \mathbb{R}$ be a function. We say $\phi$ is a \textit{Mattila--Sj\"{o}lin type function} of index $\gamma$ if $\gamma$ is the smallest number satisfying the property that for any compact set $E\subset \mathbb{R}^d$, %we have
$\phi(E, E)$ has a non-empty interior whenever $\dim_H(E)>\gamma$. The usual distance function, $\phi(x, y)=|x-y|$, is conjectured to be a Mattila--Sj\"{o}lin type function of index $\frac{d}{2}$. In the setting of finite fields $\mathbb{F}_q$, this definition is equivalent to the statement that $\phi(E, E)=\mathbb{F}_q$ whenever $|E|\gg q^{\gamma}$. The main purpose of this paper is to prove the existence of such functions with index $\frac{d}{2}$ in the vector space $\mathbb{F}_q^d$.
\end{abstract}
\section{Introduction}
For $E\subset \mathbb{R}^d$, we define its distance set by $\Delta(E):=\{|x-y|\colon x, y\in E\}$.
The classical Falconer distance conjecture \cite{falconer} says that for any compact set $E\subset \mathbb{R}^d$, if the Hausdorff dimension { $\dim_H(E)$ of $E$ is greater than $\frac{d}{2}$, then the Lebesgue measure $\mathcal{L}(\Delta(E))$ of %the distance set
$\Delta(E)$ is positive.}

The first result was given by Falconer \cite{falconer} in 1983, which says that if $\dim_H(E)>\frac{d+1}{2}$, then $\mathcal{L}(\Delta(E))>0$. %where we denote the Hausdorff dimension of a set $E$ by  $\dim_H(E)$  and the Lebesgue measure of the distance set by $\mathcal{L}(\Delta(E))$.
The threshold $\frac{d+1}{2}$ was improved to $\frac{d}{2}+\frac{1}{3}$ by Wolff \cite{37} in 1999 for $d=2$ and Erdo\u{g}an \cite{Erdogan} in 2005 for $d\ge 3$. The followings are the current best results:
\begin{itemize}
\item $d=2$, Guth, Iosevich, Ou, and Wang \cite{alex-fal} (2019): $1+\frac{1}{4}$
\item $d=3$, Du, Guth, Ou, Wang, Wilson, and Zhang \cite{Du1} (2017): $\frac{3}{2}+\frac{3}{10}$
\item $d\ge 4$ even, Du, Iosevich, Ou, Wang, and Zhang \cite{Du3} (2020): $\frac{d}{2}+\frac{1}{4}$
\item $d\ge 5$ odd, Du and Zhang \cite{Du2} (2018): $\frac{d}{2}+\frac{d}{4d-2}$
\end{itemize}
In another direction, Mattila and
Sj\"{o}lin \cite{MS} obtained a stronger conclusion under the same condition $\dim_H(E)>\frac{d+1}{2}$, namely, one has that the distance set has non-empty interior. It has also been conjectured that the right dimension should be $\frac{d}{2}$,  see \cite[Conjecture 4.4.]{M}. Several extensions of this result for general functions and configurations have been obtained recently, for instance, see \cite{GIK, GIK1}. In a very recent paper, the second, third, and fourth listed authors obtained an improvement for Cartesian product sets, namely, $E=A^d\subset \mathbb{R}^d$. In particular, we have
\begin{theorem}[\cite{koh}] Let $A$ be a compact set in $\mathbb R.$ Then we have $Int(\Delta(A^d))\ne \emptyset$ provided that
$$ \dim_H(A) > \left\{ \begin{array}{ll} \frac{d+1}{2d} \quad &\mbox{if}~~ 2\le d\le 4\\
\frac{d+1}{2d}-\frac{d-4}{2d(3d-4)}\quad &\mbox{if}~~ 5\le d\le 26\\
\frac{d+1}{2d}-\frac{23d-228}{114d(d-4)}\quad &\mbox{if}~~ 27\le d. \end{array} \right.$$
\end{theorem}
Let $\phi(x, y)\colon \mathbb{R}^d\times \mathbb{R}^d\to \mathbb{R}$ be a function in $2d$ variables. We say that $\phi$ is a \textit{Mattila--Sj\"{o}lin type function} of index $\gamma$ if $\gamma$ is the smallest number satisfying the property that for any compact set {$E\subset \mathbb{R}^d$, %we have
$Int(\phi(E, E))$ } contains an interval whenever the Hausdorff dimension of $E$ is greater than $\gamma$, where $\phi(E, E):=\{\phi(x, y)\colon x, y\in E\}$. Therefore, it follows directly from Mattila and
Sj\"{o}lin's result that the distance function is of index at most $\frac{d+1}{2}$. It is conjectured that its index should be as small as $\frac{d}{2}$.

Let $\mathbb{F}_q$ be a finite field of order $q$ which is a prime power. The distance function between two points $x$ and $y$ in the space $\mathbb{F}_q^d$ is defined by $||x-y||:=(x_1-y_1)^2+\cdots+(x_d-y_d)^2$. The finite field model of the Falconer distance problem was studied by Iosevich and Rudnev in \cite{IR07}. More precisely, they proved that for any $E\subset \mathbb{F}_q^d$, if $|E|\ge 4q^{\frac{d+1}{2}}$, then $\Delta(E)=\mathbb{F}_q$, which is directly in line with the Mattila-Sj\"{o}lin's result. Notice that in the continuous setting, if  $\dim_H(E)>\alpha$ implies $\mathcal{L}^1(\Delta(E))>0$ or $Int(\Delta(E))\ne \emptyset$, then it is expected in finite fields that if $|E|\gg q^{\alpha}$, then $\Delta(E)$ covers a positive proportion of all distances, or $\Delta(E)=\mathbb{F}_q$, respectively. {Comparing the finite field and continuous settings,} there is a crucial difference when $q\equiv 1\mod 4$ or $d-1=4k$, $k\in \mathbb{N}$, namely, under one of those conditions, one can easily construct \textit{null-subspaces} $V$ of reasonably large dimensions, i.e. subspaces $V$ with $||x||=0, x\cdot y=0$ for all $x, y\in V$, which of course does not exist in $\mathbb{R}^d$. As a consequence, one can construct sets $E\subset \mathbb{F}_q^d$ which have highly arithmetic structures, and for these sets, it has been indicated in \cite{HIKR10} that the exponent $\frac{d+1}{2}$ is best possible. However, {in case  $q\equiv 3\mod 4$,} if $d\ge 2$ is even, or $d=4k-1$, $k\in \mathbb{N}$, the two settings follow in the same way, and it has been conjectured that the right exponent should be $\frac{d}{2}$ for a positive proportion of all distances, or even the whole field. This conjecture is still wide open. In the setting of prime fields, the  current best exponent in the plane is $\frac{5}{4}$ due to Murphy, Petridis, Pham, Rudnev, and Stevens \cite{MPT}. We refer the reader to \cite{koh1, MPT} for recent progress.

In the setting of finite fields,  we say that a function $\phi(x, y)\colon \mathbb{F}_q^d\times \mathbb{F}_q^d\to \mathbb{F}_q$ is a Mattila--Sj\"{o}lin type function of index $\gamma$ in $\mathbb{F}_q^d$ if $\gamma$ is the smallest number satisfying the property that for any $E\subset \mathbb{F}_q^d$ with $|E|\gg q^{\gamma}$, we have $\phi(E, E)=\mathbb{F}_q$.

In both the finite field (with some conditions) and continuous settings, the distance function is conjectured to be of index $\frac{d}{2}$. However, to the best knowledge of the authors, we are not aware of any function of index $\frac{d}{2}$. The main purpose of this paper is to prove the existence of such a function in vector spaces over finite fields.

To state our main theorems, we need the following definition.
\begin{definition}\label{def} Let $d=2n$ for a positive integer $n.$
For each $x=(x_1, x_2, \ldots, x_d) \in \mathbb F_q^d,$  we denote
$x'=(x_1,x_2, \ldots, x_n)$ and $x''=(x_{n+1}, x_{n+2}, \ldots, x_d).$
For each $x, y \in \mathbb F_q^d,$ we define
$$ \phi(x, y):=\left\{ \begin{array}{ll}\frac{||x'-y'||}{||x''-y''||} \quad &\mbox{if}~~ ||x''-y''|| \ne 0\\
0 \quad &\mbox{if}~~ ||x''-y''||=0.\end{array}\right.$$
\end{definition}

In the following theorems, we will see that the function $\phi$ defined as in Definition \ref{def} is a Mattila--Sj\"{o}lin type function of index $\frac{d}{2}$ if and only if $q\equiv 3\mod 4$ and $d=4$. Unlike the case of usual distance function, Theorems \ref{main1} and \ref{main22} provide a complete description in the setting of finite fields.
\begin{theorem}\label{main1}
Let $q\equiv 3 \mod{4}.$ If $E\subset \mathbb F_q^4$ and $|E|> q^2$, then
$$ \phi(E, E)=\mathbb F_q.$$
\end{theorem}
The sharpness of this theorem can be checked easily. For example,
if we take $E=\mathbb F_q^2 \times \{(0,0)\},$ then $|E|=q^2$ and $\phi(E, E)=\{0\}.$
Hence, when $d=4$, the exponent $d/2$ cannot be improved. Moreover,  the assumption that $q\equiv 3 \mod{4}$ cannot be dropped, since  otherwise there is an element $i\in \mathbb F_q$ such that $i^2=-1$ so that  one can take $E=\mathbb F_q^2 \times \{(t, it): t\in \mathbb F_q\}$ for which $|E|=q^3$ and $\phi(E, E)=\{0\}.$

For all cases other than those specified in Theorem \ref{main1}, it turns out that the $\frac{d}{2}$-exponent cannot be attained. More precisely, we have the following theorem, which is also optimal up to a constant factor.
\begin{theorem}\label{main22}
 Let $E$ be a subset of $\mathbb F_q^d$.
 \begin{enumerate}
 \item %Suppose that 
 {If $d=4k$ for some $k\in \mathbb N$, and $|E|\ge C q^{\frac{3d}{4}}$, then  $\phi(E, E)=\mathbb F_q.$ Furthermore,  if $k$ is odd and $q\equiv 3\mod 4$, then the condition $|E|\ge C q^{\frac{3d-4}{4}}$ is enough.}
 \item {If $d=4k+2$ and for some $k\in \mathbb N$,} and $|E|\ge C q^{\frac{3d-2}{4}}$, then  $\phi(E, E)=\mathbb F_q.$
 \end{enumerate}
\end{theorem}
It follows from Theorems \ref{main1} and \ref{main22} that one can expect a function to be a Mattila--Sj\"{o}lin type function of index $\frac{d}{2}$ in $\mathbb{R}^d$ only in some specific dimensions. This leads us to the following conjecture which will be addressed in a sequel paper.
\begin{conjecture}
Let $\phi\colon \mathbb{R}^d\times \mathbb{R}^d\to \mathbb{R}$ be a function defined as in Definition \ref{def}. We have $\phi$ is a Mattila--Sj\"{o}lin type function of index $\frac{d}{2}$ if and only if $d=4$.
\end{conjecture}
\begin{remark}
We note here that when $E=A\times A$ for $A\subset \mathbb{F}_q^n$ or $E=A^{2n}$ for $A\subset \mathbb{F}_q$, the cardinality of $\phi(E, E)$ has been studied in earlier papers  by Iosevich, Koh, and Parshall \cite{IKP19}, Pham and Suk \cite{phamsuk}, respectively, with different approaches. More precisely, the method in \cite{IKP19} only works for the case $E=A\times A$, $A\subset \mathbb{F}_q^n$, and the key idea in the approach of \cite{phamsuk} comes from arithmetic structures of sets. If the function $\phi(x, y)$ is defined by $||x'-y'||\cdot ||x''-y''||$, then the authors of \cite{IKK} proved that the index of this function is at most $\frac{d}{2}+\frac{1}{4}$ for only the family of sets $E=A\times A$, $A\subset \mathbb{F}_q^n$.
\end{remark}
\begin{remark}
In a recent paper \cite{GIK1}, Greenleaf, Iosevich, and Taylor studied several extensions of Mattila--Sj\"{o}lin's theorem. {More precisely,} they showed that for several families of functions $\phi\colon \mathbb{R}^d\times \cdots \times \mathbb{R}^d\to \mathbb{R}^m$, for any $E\subset \mathbb{R}^d$, if $\dim_H(E)>\frac{d+1}{2}$, then the interior of $\phi(E, \cdots, E)$ is non-empty. Our definition of Mattila--Sj\"{o}lin type functions can also be extended in this general setting.
\end{remark}
\section{Preliminaries and key lemmas}
We first start by recalling notations from Fourier analysis over finite fields from \cite{IK04, LN97}.  The Fourier transform of $f$, denoted by $\widehat{f},$ is defined by
$$ \widehat{f}(m):=q^{-d} \sum_{x\in \mathbb F_q^n} \chi(-m\cdot x) f(x),$$ Let $f$ be a complex valued function on $\mathbb F_q^n.$
Here, and throughout this paper,  $\chi$ denotes
 the canonical additive character of $\mathbb F_q.$

 The Fourier inversion theorem is given by
$$ f(x)=\sum_{m\in \mathbb F_q^n} \chi(m\cdot x) \widehat{f}(m).$$
The orthogonality of the additive character $\chi$ says that
$$ \sum_{\alpha\in \mathbb F_q^n} \chi(\beta\cdot \alpha)
=\left\{\begin{array}{ll} 0\quad &\mbox{if}\quad \beta\ne (0,\ldots, 0),\\
q^d\quad &\mbox{if}\quad \beta=(0,\ldots,0). \end{array}\right.$$

By the orthogonality of $\chi$, it is not hard to prove that
$$ \sum_{m\in \mathbb F_q^n} |\widehat{f}(m)|^2 =q^{-d}\sum_{x\in \mathbb F_q^n} |f(x)|^2.$$
This formula is referred to as the Plancherel theorem. For example, it is a direct consequence of the Plancherel theorem  that for any set $E$ in $\mathbb F_q^n,$
$$ \sum_{m\in \mathbb F_q^n} |\widehat{E}(m)|^2= q^{-n}|E|.$$
Here, and throughout this paper,  we identify a set $E$ with the indicator function $1_E$ on $E.$

Throughout this paper,  we denote by $\eta$  the quadratic character of $\mathbb F_q$ with the convention that $\eta(0)=0.$ More precisely, we take $\eta(t)=1$ for a square number $t$ in $\mathbb F_q^*$, and $-1$ otherwise.   For a non-zero element $a \in \mathbb F_q^*,$ the Gauss sum $\mathcal{G}_a$ is defined by
$$\mathcal{G}_a:=\sum_{s\in \mathbb F_q^*}\eta(s) \chi(as).$$
The Gauss sum $\mathcal{G}_a$ is also written as follows: for $a\in \mathbb F_q^*,$
$$ \mathcal{G}_a=\sum_{s\in \mathbb F_q} \chi(as^2)=\eta(a) \mathcal{G}_1.$$
The explicit form of $\mathcal{G}_a$ is presented in the following lemma.
%Throughout this paper  we will write $\mathcal{G}$ for $\mathcal{G}_1.$
\begin{lemma}[\cite{LN97}, Theorem 5.15]\label{ExplicitGauss}
Let $\mathbb F_q$ be a finite field with $ q= p^{\ell},$ where $p$ is an odd prime and $\ell \in {\mathbb N}.$
Then we have
$$\mathcal{G}_1= \left\{\begin{array}{ll}  {(-1)}^{\ell-1} q^{\frac{1}{2}} \quad &\mbox{if} \quad p \equiv 1 \mod 4 \\
{(-1)}^{\ell-1} i^\ell q^{\frac{1}{2}} \quad &\mbox{if} \quad p\equiv 3 \mod 4.\end{array}\right. $$
\end{lemma}
There are several consequences of this lemma, which will be used {throughout} the paper. More precisely, one can prove the following estimate by completing the square and using a change of variables:
\begin{equation}\label{ComSqu}
 \sum_{s\in \mathbb F_q} \chi(as^2+bs)= \eta(a)\mathcal{G}_1 \chi\left(\frac{b^2}{-4a}\right).\end{equation}
We recall the fact that $q\equiv 3\mod 4$ if and only if $l$ is odd and $p\equiv 3\mod 4$, so
\begin{equation}\label{Corm}
 \mathcal{G}_1^2=\eta(-1) q.
\end{equation}
\subsection{Key lemmas}
In this section, we present key lemmas in the proofs of Theorems \ref{main1} and \ref{main22}.

Given an even integer $d\ge 4,$ we define
$$S_0:=\{x'\in \mathbb F_q^{d/2}: ||x'||=0\},$$
which is called the zero sphere in $\mathbb F_q^{d/2}.$
We shall invoke the well--known Fourier transform on the zero sphere $S_0$ in $\mathbb F_q^{d/2},$ which is closely related to the Gauss sum.
\begin{lemma} [\cite{IK10}, Lemma 4] \label{S_0FTF} For  $d\ge 4$ even, let  $S_0$ be the zero sphere in $\mathbb F_q^{d/2}.$
Then, for $m'\in \mathbb F_q^{d/2}$, we have
$$ \widehat{S_0}(m')= \frac{\delta_{\mathbf{0}}(m')}{q} + q^{-(d+2)/2} \eta^{d/2}(-1) \mathcal{G}^{d/2}_1 \sum_{r\in \mathbb F_q^*} \eta^{d/2}(r) \chi(r||m'||),$$
where $\delta_{\mathbf{0}}(x)=1$ if $x={\mathbf{0}}$, and $0$ otherwise.
\end{lemma}

Using the explicit value of the Gauss sum $\mathcal{G}_1$,  Lemma \ref{S_0FTF} yields the following corollaries.
\begin{corollary}\label{cor2.4} For $d/2\ge 2$ even,  let $S_0$ be the zero sphere in $\mathbb F_q^{d/2}$ and let $m'\in \mathbb F_q^{d/2}.$ Then the following statements hold.
\begin{enumerate}
\item If $d/2=4k$, $k\in \mathbb N,$ then
$$\widehat{S_0}(m')= \left\{\begin{array}{ll}  q^{-1}+  q^{-d/4}- q^{-(d+4)/4} \quad &\mbox{if} \quad m'=(0,\ldots,0)\\
q^{-d/4}- q^{-(d+4)/4}  \quad &\mbox{if} \quad ||m'||=0, m'\ne (0,\ldots,0)\\
 - q^{-(d+4)/4}  \quad &\mbox{if} \quad ||m'||\ne 0.\end{array}\right. $$
\item If $d/2=4k-2$, $k\in \mathbb N,$  then
$$\widehat{S_0}(m')= \left\{\begin{array}{ll}  q^{-1}+  \eta(-1)q^{-d/4}-\eta(-1) q^{-(d+4)/4} \quad &\mbox{if} \quad m'=(0,\ldots,0)\\
\eta(-1)q^{-d/4}-\eta(-1)q^{-(d+4)/4}  \quad &\mbox{if} \quad ||m'||=0, m'\ne (0,\ldots,0)\\
  -\eta(-1)q^{-(d+4)/4}  \quad &\mbox{if} \quad ||m'||\ne 0.\end{array}\right. $$
\item If $d/2=4k-2$, $k\in \mathbb N,$ and $q\equiv 3 \mod{4},$ then
$$\widehat{S_0}(m')= \left\{\begin{array}{ll}  q^{-1}-q^{-d/4}+q^{-(d+4)/4} \quad &\mbox{if} \quad m'=(0,\ldots,0)\\
-q^{-d/4}+ q^{-(d+4)/4}  \quad &\mbox{if} \quad ||m'||=0, m'\ne (0,\ldots,0)\\
  q^{-(d+4)/4}  \quad &\mbox{if} \quad ||m'||\ne 0.\end{array}\right. $$
\end{enumerate}
\end{corollary}
\begin{proof}
Since $d/2$ is even, we have $\eta^{d/2}\equiv 1.$ By Lemma \ref{S_0FTF},
$$ \widehat{S_0}(m')= \frac{\delta_{\mathbf{0}}(m')}{q} + q^{-(d+2)/2}  \mathcal{G}^{d/2}_1 \sum_{r\in \mathbb F_q^*}  \chi(r||m'||).$$
By the orthogonality of $\chi,$  we notice that
$$\sum_{r\in \mathbb F_q^*}  \chi(r||m'||)=q \delta_0(||m'||) -1.$$
So we have
\begin{equation}\label{eq2.2K}  \widehat{S_0}(m')= \frac{\delta_{\mathbf{0}}(m')}{q} + q^{-d/2}  \mathcal{G}^{d/2}_1 \delta_0(||m'||)-q^{-(d+2)/2}  \mathcal{G}^{d/2}_1.\end{equation}
By using (\ref{Corm}), one has
$$ \mathcal{G}^{d/2}_1 =(\eta(-1))^{d/4}q^{d/4}.$$

In Case (1), since $d/4$ is even, $\mathcal{G}^{d/2}_1 =q^{d/4}.$ In Case (2), since $d/4$ is odd, we have $\mathcal{G}^{d/2}_1 =\eta(-1)q^{d/4}.$
In Case (3),  since $d/4$ is odd and $q\equiv 3 \mod{4}$ (namely, $\eta(-1)=-1$), we have $\mathcal{G}^{d/2}_1 =-q^{d/4}.$ Hence, the proof is complete.
\end{proof}
\begin{corollary}\label{cor2.5}
For $d/2\ge 3$ odd,  let $S_0$ be the zero sphere in $\mathbb F_q^{d/2}$ and let $m'\in \mathbb F_q^{d/2}.$ Then the following statements hold.
\begin{enumerate}
\item If $d/2=4k-1$, $k\in \mathbb N,$ then
$$\widehat{S_0}(m')= \left\{\begin{array}{ll}  q^{-1} \quad &\mbox{if} \quad m'=(0,\ldots,0)\\
0  \quad &\mbox{if} \quad ||m'||=0, m'\ne (0,\ldots,0)\\
  q^{-(d+2)/4} \eta(-||m'||)  \quad &\mbox{if} \quad ||m'||\ne 0.\end{array}\right. $$
\item If $d/2=4k+1$, $k\in \mathbb N,$ then
$$\widehat{S_0}(m')= \left\{\begin{array}{ll}  q^{-1} \quad &\mbox{if} \quad m'=(0,\ldots,0)\\
0  \quad &\mbox{if} \quad ||m'||=0, m'\ne (0,\ldots,0)\\
q^{-(d+2)/4} \eta(||m'||) \quad &\mbox{if} \quad ||m'||\ne 0.\end{array}\right. $$
\end{enumerate}
\end{corollary}
\begin{proof}
Since $d/2$ is odd,  $\eta^{d/2}=\eta.$
By Lemma \ref{S_0FTF},
$$\widehat{S_0}(m')= \frac{\delta_{\mathbf{0}}(m')}{q} + q^{-(d+2)/2} \eta(-1) \mathcal{G}^{d/2}_1 \sum_{r\in \mathbb F_q^*} \eta(r) \chi(r||m'||).$$
Since $\sum_{r\in \mathbb F_q^*} \eta(r) \chi(r||m'||) = \eta(||m'||) \mathcal{G}_1,$  we have
$$\widehat{S_0}(m')=\frac{\delta_{\mathbf{0}}(m')}{q} + q^{-(d+2)/2} \eta(-1) \eta{(||m'||)}\mathcal{G}^{(d+2)/2}_1  .$$
By using (\ref{Corm}),
$$ \mathcal{G}^{(d+2)/2}_1= (\eta(-1))^{(d+2)/4} q^{(d+2)/4}.$$
In Case (1),  $\mathcal{G}^{(d+2)/2}_1=q^{(d+2)/4}$, since $(d+2)/4$ is even. On the other hand,  in Case (2), since $(d+2)/4$ is odd,  we have $\mathcal{G}^{(d+2)/2}_1=\eta(-1)q^{(d+2)/4}.$ Hence, the proof follows.
\end{proof}
%%%%%%%%%%%%%%%%%%%%%%%%%%%%%%%%%%%%%%%%%%%%%%%%
%%%%%%%%%%%%%%%%%%%%%%%%%%%%%%%%%%%%%%%%%%%%%%%

For $t\in \mathbb F_q,$ we define
$$ R_t:=\{x\in \mathbb F_q^d: \phi(x, 0)=t\},$$
which can be viewed as the ``sphere"  centered at the origin of radius $t$ with respect to the function $\phi$.

The Fourier transform on the set $R_t, t\ne 0,$ takes the following form, which plays a crucial role in proving  main theorems.
\begin{lemma} \label{lem2.1} For $t\ne 0$, we have {
$$ \widehat{R_t}(m)= q^{-1} \delta_{\mathbf{0}}(m)-\widehat{S_0}(m') \widehat{S_0}(m'') +  {\mathcal{G}^d_1} q^{-d-1} \eta^{d/2}(-t) \left(q\delta_{\mathbf{0}}(t||m'||-||m''||)-1\right) .$$}
%where $\eta$ denotes the quadratic character of $\mathbb F_q$, and $\delta_{\mathbf{0}}(x)=1$ if $x={\mathbf{0}}$, and $0$ otherwise.
\end{lemma}
\begin{proof}
By the definition, we have
$$ \widehat{R_t}(m)=q^{-d} \sum_{x\in \mathbb F_q^d: \phi(x, 0)=t} \chi(-m\cdot x).$$
Since $t\ne 0$, we see
$$\widehat{R_t}(m)=q^{-d} \sum_{\substack{x\in \mathbb F_q^d:\\ ||x''||\ne 0, \phi(x, 0)=t}} \chi(-m\cdot x)=q^{-d} \sum_{\substack{x\in \mathbb F_q^d:\\ ||x'||-t||x''||=0}} \chi(-m\cdot x) - q^{-d} \sum_{\substack{x\in \mathbb F_q^d:\\ ||x'||=||x''||=0}} \chi(-m\cdot x).$$
By the definition of the Fourier transform of the indicator function on the zero sphere $S_0$ in $\mathbb F_q^{d/2},$  the last term above equals
$ -\widehat{S_0}(m') \widehat{S_0}(m'').$ Thus to complete the proof, it suffices to prove that
$$ \mathbf{M}:= q^{-d} \sum_{\substack{x\in \mathbb F_q^d:\\ ||x'||-t||x''||=0}} \chi(-m\cdot x) = q^{-1} \delta_{\mathbf{0}}(m)+  {\mathcal{G}^d_1} q^{-d-1} \eta^{d/2}(-t) \left(q\delta_{\mathbf{0}}(t||m'||-||m''||)-1\right).$$

To prove this, we apply the orthogonality of $\chi$ and the Gauss sum estimate. It follows that
$$ \mathbf{M}= q^{-d} \sum_{x\in \mathbb F_q^d} q^{-1}\sum_{ s\in \mathbb F_q}  \chi(s(||x'||-t||x''||)) \chi(-m\cdot x).$$
Considering  the cases of $s=0$ and $s\ne 0$,  we have
$$ \mathbf{M}= q^{-1}\delta_{\mathbf{0}}(m) + q^{-d-1} \sum_{s\ne 0} \sum_{x\in \mathbb F_q^d} \chi(s||x'||-m'\cdot x') \chi(-st||x''||-m''\cdot x'').$$
Applying the formula \eqref{ComSqu},  we obtain
$$ \mathbf{M}=q^{-1}\delta_{\mathbf{0}}(m) +  q^{-d-1} \mathcal{G}_1^d  \eta^{d/2}(-t) \sum_{s\ne 0} \chi\left( \frac{t||m'||-||m''||}{-4st}\right).$$
By the orthogonality of $\chi$, it is not hard to check that
$$\sum_{s\ne 0} \chi\left( \frac{t||m'||-||m''||}{-4st}\right) =q\delta_{\mathbf{0}}(t||m'||-||m''||)-1.$$  Therefore, the proof is complete.
\end{proof}

\section{Proofs of main results}
We proceed with the counting function argument.
For $t\in \mathbb F_q$, let $\nu(t)$ be the number of pairs $(x,y)$ in $E\times E$ such that
$\phi(x, y)=t.$ Since $0\in \phi(E, E)$ for any $E\subset \mathbb{F}_q^d$, our task is to find size condition of a set $E$ in $\mathbb F_q^d$ such that  $\nu(t)>0$ for any nonzero $t$ in $\mathbb F_q.$ Fix a nonzero $t \in \mathbb F_q^*$.  { Then  it is immediate that
$$ \nu(t)=\sum_{x,y\in \mathbb F_q^d} E(x) E(y) R_t(x-y).$$}
%where $R_t=\{x\in \mathbb F_q^d: \phi(x, 0)=t\}.$ 
Applying the Fourier inversion theorem to the function $R_t(x-y)$,  we have
$$\nu(t) = q^{2d} \sum_{m\in \mathbb F_q^d} \widehat{R_t}(m) |\widehat{E}(m)|^2.$$
Combining with Lemma \ref{lem2.1}, we see from a direct computation  that
\begin{align}\label{eq3.1} \nu(t)=&q^{-1}|E|^2 - q^{2d} \sum_{m\in \mathbb F_q^d} \widehat{S_0}(m') \widehat{S_0}(m'') |\widehat{E}(m)|^2\\
&+ q^d {\mathcal{G}_1}^d {\eta}^{d/2}(-t) \sum_{t||m'||-||m''||=0} |\widehat{E}(m)|^2 -q^{-1}{\mathcal{G}_1}^d  \eta^{d/2}(-t) |E|\nonumber.\end{align}

\subsection{Proof of Theorem \ref{main1}}
Since $d=4$, it is clear that ${\mathcal{G}}^d_1=q^{d/2}=q^2$ and $\eta^{d/2}\equiv 1.$ Furthermore, since $-1$ is not a square number, we have $S_0=\{(0,0)\}$ and  so $\widehat{S_0}(m')=q^{-2}=\widehat{S_0}(m'')$ for all $m\in \mathbb F_q^4.$
Hence, it follows from \eqref{eq3.1} that
$$ \nu(t)=q^{-1}|E|^2 - q^{8} q^{-4} \sum_{m\in \mathbb F_q^4} |\widehat{E}(m)|^2
+ q^4 q^2  \sum_{t||m'||-||m''||=0} |\widehat{E}(m)|^2 -q^{-1}q^2 |E|.$$
To compute the above second term,
we recall from the Plancherel theorem that $\sum_{m\in \mathbb F_q^4} |\widehat{E}(m)|^2=q^{-4}|E|.$
The above third term is estimated as follows:
$$q^4 q^2  \sum_{t||m'||-||m''||=0} |\widehat{E}(m)|^2 \ge q^4q^2 |\widehat{E}{(0,0,0,0)}|^2=q^{-2}|E|^2.$$

Therefore, we have
$$ \nu(t)\ge  q^{-1}|E|^2 -|E|+ q^{-2}|E|^2-q|E|=
(q^{-1}+q^{-2}) |E|(|E|-q^2).$$
This implies that if $|E|> q^2$ then $\nu(t)>0.$  This completes the proof of Theorem \ref{main1}.

\subsection{Proof of Theorem \ref{main22}(1)}
It follows from the statement of Theorem \ref{main22}(1) that we need to prove two separated cases:
\begin{enumerate}
\item[\textbf{$\mathtt{\bf A}$}.]  $d=8k+4$, $k\in \mathbb{N}$, and $q\equiv 3\mod 4$.
\item[\textbf{$\mathtt{\bf B}$}.] $d=4k$, $k\in \mathbb{N}$.
\end{enumerate}
We remark here that the case $\mathtt{\bf A}$ is an extension of Theorem \ref{main1}.

\begin{proof}[Proof of Case $\mathtt{\bf A}$]
From our assumption that $d=8k+4$, $k\in \mathbb N,$ it is clear that $\mathcal{G}^d_1=q^{d/2}$ and $\eta^{d/2}\equiv 1.$ Therefore, the identity \eqref{eq3.1} becomes

\begin{align}\label{eq3.2N}\nu(t)=&q^{-1}|E|^2 - q^{2d} \sum_{m\in \mathbb F_q^d} \widehat{S_0}(m') \widehat{S_0}(m'') |\widehat{E}(m)|^2\\
&+ q^{3d/2}  \sum_{t||m'||-||m''||=0} |\widehat{E}(m)|^2 -q^{(d-2)/2} |E|\nonumber.\end{align}

First notice from our hypothesis that we can apply the third part of Corollary \ref{cor2.4}, that is
$$\widehat{S_0}(m')= \left\{\begin{array}{ll}  q^{-1}-q^{-d/4}+q^{-(d+4)/4} \quad &\mbox{if} \quad m'=(0,\ldots,0)\\
-q^{-d/4}+ q^{-(d+4)/4}  \quad &\mbox{if} \quad ||m'||=0, m'\ne (0,\ldots,0)\\
  q^{-(d+4)/4}  \quad &\mbox{if} \quad ||m'||\ne 0.\end{array}\right. $$

Since we aim to find a lower bound of $\nu(t)$, we may ignore some positive terms appearing when we estimate $\nu(t)$. Hence, we break down the sum $\sum_{m\in \mathbb F_q^d}$ in (\ref{eq3.2N}) into 9 subsummands:
$$\sum_{m', m''=\mathbf{0}}+ \sum_{m'=\mathbf{0}, ||m''||=0, m'\ne \mathbf{0}}+\sum_{m'=\mathbf{0}, ||m''||\ne 0} +\sum_{||m'||=0, m'\ne \mathbf{0}, m''=\mathbf{0}}$$
$$+\sum_{||m'||=0, m'\ne \mathbf{0},||m''||=0, m''\ne \mathbf{0}} +\sum_{||m'||=0, m'\ne \mathbf{0}, ||m''|| \ne 0}+ \sum_{||m'||\ne 0, m''=\mathbf{0}} +
 \sum_{||m'||\ne 0, ||m''||=0, m''\ne \mathbf{0}}+ \sum_{||m'||\ne 0, ||m''||\ne\mathbf{0}},$$
 and then we only compute such sums for which
 $\widehat{S_0}(m') \widehat{S_0}(m'')$ takes a positive value, which can be easily evaluated by using the above explicit value of $\widehat{S_0}.$
Notice that it will be enough to consider the dominant main term of $\widehat{S_0}(m').$ Namely, we can use the following approximate value of
$\widehat{S_0}$:
$$\widehat{S_0}(m')\approx \left\{\begin{array}{ll}  q^{-1} \quad &\mbox{if} \quad m'=(0,\ldots,0)\\
-q^{-d/4}  \quad &\mbox{if} \quad ||m'||=0, m'\ne (0,\ldots,0)\\
  q^{-(d+4)/4}  \quad &\mbox{if} \quad ||m'||\ne 0.\end{array}\right. $$

In addition,  notice that  the third term in \eqref{eq3.2N} is greater than or equal to the value
$$ q^{3d/2}\sum_{||m'||=0,||m''||=0} |\widehat{E}(m)|^2.$$
Then, it is not hard to obtain that
\begin{align*}\nu(t)\ge &q^{-1}|E|^2 - q^{2d-2} \sum_{m', m''=\mathbf{0}}|\widehat{E}(m)|^2 - q^{(7d-8)/4} \sum_{m'=\mathbf{0}, ||m''||\ne 0}
|\widehat{E}(m)|^2\\
&+q^{3d/2} \sum_{||m'||=0, m'\ne\mathbf{0},||m''||=0, m''\ne\mathbf{0}} |\widehat{E}(m)|^2
-q^{(7d-8)/4} \sum_{ ||m'||\ne 0, m''=\mathbf{0},}
|\widehat{E}(m)|^2  \\
&-q^{(3d-4)/2} \sum_{ ||m'||\ne 0, ||m''||\ne 0}
|\widehat{E}(m)|^2
+ q^{3d/2}\sum_{||m'||=0,||m''||=0} |\widehat{E}(m)|^2 -q^{(d-2)/2} |E|\nonumber.\end{align*}

To find a concrete lower bound of the above terms,  we first note that the second term above equals
$-q^{2d-2} |\widehat{E}(0,\ldots,0)|^2,$ which is $q^{-2}|E|^2.$ We also observe that the fourth and  seventh terms above can be ignored since they are positive,  the sum of the third and fifth terms is greater than or equal to  the value $-q^{(7d-8)/4}\sum_{m\in \mathbb F_q^d} |\widehat{E}(m)|^2.$ For other summations, we dominate them by using the quantity $\sum_{m\in \mathbb F_q^d} |\widehat{E}(m)|^2 = q^{-d}|E|.$
In other words,
$$\nu(t) \ge q^{-1}|E|^2-q^{-2}|E|^2-q^{(3d-8)/4} |E|-q^{(d-4)/2} |E|-q^{(d-2)/2}|E|.$$
By a direct computation, we conclude that if $|E|\ge C q^{(3d-4)/4}$, then $\nu(t)>0,$
as required.
\end{proof}

\begin{proof}[Proof of Case $\mathtt{\bf B}$]
Since $d=4k$,  $k\in \mathbb N,$  it is not hard to see that
${\mathcal{G}}^d_1=q^{d/2}$ and $\eta^{d/2}\equiv 1.$
Therefore, the equation \eqref{eq3.1} can be rewritten by
\begin{align*} \nu(t)=&q^{-1}|E|^2 - q^{2d} \sum_{m\in \mathbb F_q^d} \widehat{S_0}(m') \widehat{S_0}(m'') |\widehat{E}(m)|^2\\
&+ q^d q^{d/2}\sum_{t||m'||-||m''||=0} |\widehat{E}(m)|^2 -q^{-1}q^{d/2} |E|.\end{align*}
Since the third term is positive, we obtain that
$$ \nu(t)> q^{-1}|E|^2 - q^{2d} \sum_{m\in \mathbb F_q^d} |\widehat{S_0}(m')| |\widehat{S_0}(m'')| |\widehat{E}(m)|^2-q^{-1}q^{d/2} |E|.$$
We estimate the sum in $m\in \mathbb F_q^d$ by finding an upper bound of each of four subsummands:
\begin{equation}\label{eq3.2}\sum_{m', m''=(0,\ldots,0)} + \sum_{m'=(0,\ldots,0), m''\ne (0,\ldots,0)} + \sum_{m'\ne (0,\ldots,0), m''=(0,\ldots, 0) } + \sum_{ m', m''\ne (0,\ldots, 0)}.\end{equation}
To bound each of them, we apply the following facts which are direct consequences of Corollary \ref{cor2.4} (1)--(2):
$$ |\widehat{S_0}(0,\ldots, 0)| \le 2 q^{-1} \quad \mbox{and} \quad  \max_{m'\ne (0,\ldots,0)} |\widehat{S_0}(m')| \le q^{-d/4}.$$
By a simple algebra, we then obtain that
\begin{align*} \nu(t)>& q^{-1}|E|^2 - 4q^{2d-2} |\widehat{E}(\mathbf{0},\mathbf{0} )|^2 -2q^{(7d-4)/4} \left[\sum_{m''\ne \mathbf{0}}  |\widehat{E}(\mathbf{0}, m'')|^2+\sum_{m'\ne \mathbf{0}}  |\widehat{E}(m', \mathbf{0})|^2 \right]\\
&-q^{3d/2} \sum_{ m', m''\ne \mathbf{0}} |\widehat{E}(m', m'')|^2 -q^{-1}q^{d/2} |E|.\end{align*}
Notice that $|\widehat{E}(\mathbf{0},\mathbf{0} )|=q^{-d}|E| = \sum_{m\in \mathbb F_q^d} |\widehat{E}(m)|^2.$ Also observe that both the sums in the above bracket and the above sum over $m', m''\ne \mathbf{0}$ are dominated by the sum $\sum_{m\in \mathbb F_q^d} |\widehat{E}(m)|^2.$ From these observations, we have
$$ \nu(t)> q^{-1}|E|^2 -4 q^{-2}|E|^2 -2q^{(3d-4)/4} |E| -q^{d/2}|E| -q^{(d-2)/2} |E|$$
$$\ge q^{-1}|E|^2 -4 q^{-2}|E|^2 -4q^{(3d-4)/4} |E|.$$
We may assume that $q$ is {sufficiently large. For,  otherwise } we can take $C>0$ such that $q^d=Cq^{3d/4}$ for which $|E|=q^d$ and $\phi(E, E)=\mathbb F_q.$   Hence, we conclude that if $|E|\ge C q^{3d/4}$ for a sufficiently large constant $C>0$, then $\nu(t)>0$.
This completes the proof. \end{proof}

\subsection{Proof of Theorem \ref{main22}(2)} The proof is almost identical with that of Theorem \ref{main22}(1). The main difference is that since $d=4k+2$, $k\in \mathbb N,$ we can invoke Corollary \ref{cor2.5}, which gives much better  Fourier decay on the zero sphere $S_0$ than Corollary \ref{cor2.4}. For the sake of completeness,  we now give a detailed proof.  Since $|\mathcal{G}^d_1|=q^{d/2},$  the equation \eqref{eq3.1} implies that
\begin{align*} \nu(t)\ge &q^{-1}|E|^2 - q^{2d} \sum_{m\in \mathbb F_q^d} |\widehat{S_0}(m')| |\widehat{S_0}(m'')| |\widehat{E}(m)|^2\\
&- q^{3d/2}  \sum_{m\in \mathbb F_q^d} |\widehat{E}(m)|^2 -q^{(d-2)/2} |E|\nonumber.\end{align*}
  We estimate the second term above by decomposing  it into four subsummands as in \eqref{eq3.2}.
  In addition, we use the following Fourier decay estimates on $S_0$, which follow immediately from Corollary \ref{cor2.5}:
  $$ |\widehat{S_0}(0,\ldots, 0)| = q^{-1} \quad \mbox{and} \quad  \max_{m'\ne (0,\ldots,0)} |\widehat{S_0}(m')| \le q^{-(d+2)/4}.$$
We then have
\begin{align*} \nu(t)>& q^{-1}|E|^2 - q^{2d-2} |\widehat{E}(\mathbf{0},\mathbf{0} )|^2 -q^{(7d-6)/4} \left[\sum_{m''\ne \mathbf{0}}  |\widehat{E}(\mathbf{0}, m'')|^2+\sum_{m'\ne \mathbf{0}}  |\widehat{E}(m', \mathbf{0})|^2 \right]\\
&-q^{(3d-2)/2} \sum_{ m', m''\ne \mathbf{0}} |\widehat{E}(m', m'')|^2 -q^{(d-2)/2} |E|.\end{align*}
As mentioned before, both the value in the above bracket and  the sum over $m',m''\ne 0$ are less than $\sum_{m\in \mathbb F_q^d} |\widehat{E}(m)|^2$ , which equals $q^{-d}|E|$.
 Also recall that $|\widehat{E}(\mathbf{0}, \mathbf{0})|=q^{-d}|E|.$
 Then we see that
 \begin{align*}\nu(t)>& q^{-1}|E|^2-q^{-2}|E|^2- q^{(3d-6)/4} |E|-q^{(d-2)/2}|E|-q^{(d-2)/2} |E|\\
 \ge& q^{-1}|E|^2-q^{-2}|E|^2- 3q^{(3d-6)/4} |E|.\end{align*}
 This clearly implies that $\nu(t)>0$ if $|E|\ge C q^{(3d-2)/4}$ for some large constant $C>0.$
 This completes the proof.

\begin{remark}
It is quite natural to ask whether our main theorems can be extended for the case of two sets $\phi(E, F)$. However, it seems that  we need to find a new approach for this question.   The main reason is that, for instance, in the proof of Theorem \ref{main1}, the sign of the term $\sum_{m\in \mathbb{F}_q^d}|\widehat{E}(m)|^2$ is negative, but for two different sets, one might get $\sum_{m\in \mathbb{F}_q^d}\widehat{E}(m)\overline{\widehat{F}}(m)$ which can be positive and quite large. This question will be addressed in a sequel paper.
\end{remark}
\section{Sharpness of Theorem \ref{main22}} \label{sec4}
To construct desired sets,  let us first recall the following lemma in a paper due to Vinh \cite{Vi12}.
\begin{lemma}[\cite{Vi12}, Lemma 2.1]\label{lem4.1}
 {Let $S_0=\{(x_1,\ldots, x_{n}): x_1^2+\cdots+x_{n}^2=0\}$ be a variety in $\mathbb F_q^{n}$ with $n\ge 2.$ %and let $\eta$ denote the quadratic character of $\mathbb F_q.$
If} $H$ is a subspace of maximal dimension contained in $S_0$, then we have the following facts:
\begin{enumerate}
\item If $n$ is odd, then $|H|=q^{\frac{n-1}{2}}$
\item If $n$ is even and $(\eta(-1))^{\frac{n}{2}}=1$, then $|H|=q^{\frac{n}{2}}$
\item If $n$ is even and $(\eta(-1))^{\frac{n}{2}}=-1,$ then $|H|=q^{\frac{n-2}{2}}.$
\end{enumerate}
\end{lemma}
By using this lemma, we are able to construct sets that meet the exponents of Theorem \ref{main22}.
\begin{lemma}[Sharpness of Theorem \ref{main22}(1)--Case $\mathtt{\bf A}$]\label{lem0}
{ Suppose $d=8k+4$ for some $k\in \mathbb N,$ and $ q\equiv 3 \mod{4}.$ Then }there is a set $E$ in $\mathbb F_q^{d}$ such that $|E|=q^{(3d-4)/4}$ and $\phi(E, E)=\{0\}.$
\end{lemma}
\begin{proof}
First, we have $\eta(-1)=-1$ since $q\equiv 3 \mod{4}.$ Let $n:=d/2=4k+2$, $k\in \mathbb N.$ We apply the third part of Lemma \ref{lem4.1} so that we can choose  a subspace $H$ in $S_0\subset \mathbb F_q^{d/2}$ with $|H|=q^{(n-2)/2}=q^{(d-4)/4}.$ Moreover, since $H-H=H$,  we have $||a-b||=0$ for all $a, b\in H$. We now define $E=\mathbb F_q^{d/2}\times H.$ It is clear that $|E|=q^{(3d-4)/4}$ and $\phi(E, E)=\{0\},$
as required.
\end{proof}
\begin{lemma}[Sharpness of Theorem \ref{main22}(1)--Case $\mathtt{\bf B}$]\label{lem4.3}
Suppose $d=4k$ { for some $k\in \mathbb{N}$. If  $k$ and $q$ satisfy one of the following two conditions:
\begin{itemize}
\item $k$ is odd and $q\equiv 1\mod 4$.
\item $k$ is even,
\end{itemize}
then} there exists a set $E$ in $\mathbb F_q^{d}$ such that $|E|=q^{3d/4}$ and $\phi(E, E)=\{0\}.$
\end{lemma}
\begin{proof} We proceed with the same argument as in the proof of Lemma \ref{lem0}. Recall that $q\equiv 1 \mod{4}$ if and only if $\eta(-1)=1.$
Put $n:=d/2.$ It follows from our conditions on $k$ and $q$ that $(\eta(-1))^{\frac{n}{2}}=1$. So, one can apply Lemma \ref{lem4.1} to obtain a subspace $H$ in $S_0\subset \mathbb F_q^{d/2}$ with $|H|=q^{n/2}=q^{d/4}.$ Setting $E=\mathbb F_q^{d/2} \times H$, the proof follows.
\end{proof}
\begin{lemma}[Sharpness of Theorem \ref{main22}(2)]
Suppose that $d=4k+2$  { for some $k\in \mathbb N$. Then }there is a set $E$ in $\mathbb F_q^d$ such that $|E|=q^{(3d-2)/4},$ and  $\phi(E, E)=\{0\}.$
\end{lemma}
\begin{proof}
 Let $n=d/2.$ Since $n$ is odd, we can use the first part of Lemma \ref{lem4.1} with $n=d/2$ so that we can find a subspace $H$ in $S_0\subset\mathbb F_q^{d/2}$ with $|H|=q^{(n-1)/2}= q^{(d-2)/4}.$ Since $H-H=H$, we have $||a-b||=0$ for any $a, b\in H$. Hence, the set $E=\mathbb F_q^{d/2}\times H$ is what we need.
\end{proof}

\section*{Acknowledgments}
The authors would like to thank two reviewers for useful comments and suggestions.

Daewoong Cheong and Doowon Koh were supported by Basic Science Research Programs through National Research Foundation of Korea (NRF) funded by the Ministry of Education (NRF-2018R1D1A3B07045594 and NRF-2018R1D1A1B07044469, respectively).
Thang Pham was supported by the Swiss National Science Foundation grants P400P2-183916 and P4P4P2-191067. Chun-Yen Shen was supported in part by MOST, through grant 108-2628-M-002-010-MY4.


\begin{thebibliography}{7}





\bibitem{CEHIK10} J. Chapman, M. Burak Erdo\u{g}an, D. Hart, A. Iosevich, and D. Koh, {\it Pinned distance sets, k-simplices, Wolff's exponent in finite fields and sum-product estimates}, Math Z. \textbf{271} (2012), no. 1, 63--93.


  \bibitem{Du1}
     X. Du, L. Guth, Y. Ou, H. Wang, B. Wilson and R. Zhang, \textit{Weighted restriction estimates and
application to Falconer distance set problem}, Amer. J. Math. (2018, to appear)
\bibitem{Du2}
 X. Du and R. Zhang, \textit{Sharp $L^2$- estimates of the Schr\"{o}dinger maximal function in higher dimensions}, Ann. of Math. \textbf{189} (2019), no. 3, 837--861.
 \bibitem{Du3}
 X. Du, A. Iosevich, Y. Ou, H. Wang, and R. Zhang, \textit{An improved result for Falconer's distance set problem in even dimensions}, arXiv:2006.06833 (2020).
 \bibitem{Erdogan}
M. B. Erdo\u{g}an, \textit{A bilinear Fourier extension theorem and applications to the distance set problem},
Int. Math. Res. Not. 2005, no. \textbf{23}, 1411--1425.
\bibitem{falconer}
K. J. Falconer, \textit{On the Hausdorff dimensions of distance sets}, Mathematika, \textbf{32} (1985), 206--212.

\bibitem{alex-fal}
L. Guth, A. Iosevich, Y. Ou, and H. Wang, \textit{On Falconer's distance set problem in the plane}, Inventiones mathematicae, \textbf{219}(3), 779--830.
\bibitem{GIK} A. Greenleaf, A. Iosevich, and K. Taylor, \emph{Configuration sets with nonempty interior,} Journal of Geometric Analysis (2019): 1--19.
\bibitem{GIK1}
A. Greeleaf, A. Iosevich, and K. Taylor, \textit{On $k$-point configuration sets with nonempty interior}, arXiv:2005.10796 (2020).
\bibitem{HIKR10} D.~ Hart, A.~ Iosevich, and D.~ Koh, M.~ Rudnev, \emph{ Averages over hyperplanes, sum-product theory in vector spaces over finite fields and the Erd\"os-Falconer distance conjecture}, Trans. Amer. Math. Soc. Volume \textbf{363}, Number 6,  (2011),  3255--3275.
\bibitem{Hegyvari}
N. Hegyv\'{a}ri and M. P\'{a}lfy, \textit{Note on a result of Shparlinski and related results}, \textbf{193} (2020), 157--163.
\bibitem{IK10} A. Iosevich and D. Koh, \emph{Extension theorems for spheres in the finite field setting,}   Forum Math. \textbf{22} (2010), no. 3, 457-483.
\bibitem{IKK}
A. Iosevich and D. Koh, \textit{On the product set of the distance set,} preprint 2017.

\bibitem{IKP19} A. Iosevich, D. Koh, and H. Parshall, \emph{On the quotient set of the distance set},  Mosc. J. Comb. Number Theory \textbf{8} (2019), no. 2, 103-115.
\bibitem{IR07} A.~ Iosevich and M. ~Rudnev, \emph{ Erd\H{o}s distance problem in vector spaces over finite fields}, Trans. Amer. Math. Soc. \textbf{359} (2007), 6127--6142.

\bibitem{IK04} H. Iwaniec and E. Kowalski, \emph{Analytic Number Theory,} Colloquium Publications 53
(2004).

\bibitem{koh} D. Koh, T. Pham, and C-Y. Shen, \textit{On the Mattila-Sj\"olin distance theorem for product sets}, submitted, 2021.
\bibitem{koh1}
D. Koh, T. Pham, and L. A. Vinh, \textit{Extension theorems and Distance problems over finite fields}, arXiv:1809.08699 (2018).

\bibitem{LN97} R. Lidl and H. Niederreiter, \emph{Finite fields,} Cambridge University Press, (1997).
\bibitem{lyall}
N. Lyall and A. Magyar, \textit{Weak hypergraph regularity and applications to geometric Ramsey theory}, to appear in Trans. Amer. Math. Soc. (2020).
\bibitem{M}
P. Mattila, \textit{Fourier analysis and Hausdorff dimension}, Vol. \textbf{150}. Cambridge University Press, 2015.

\bibitem{MS} P. Mattlia and P. Sj\"olin, \emph{Regularity of distance measures and sets,} Math. Nachr. 204 (1999), 157--162.

\bibitem{MP19}
B. Murphy and G. Petridis, \textit{An example related to the Erd\H{o}s-Falconer question over arbitrary finite fields},  Bull. Hellenic Math. Soc. \textbf{63} (2019), 38-39.

\bibitem{MPT}
B. Murphy, G. Petridis, T. Pham, M. Rudnev, and S. Stevens, \textit{On the Pinned Distances Problem over Finite Fields},  arXiv:2003.00510 (2020).

\bibitem{phamsuk}
T. Pham and A. Suk, \textit{On the structure of distance sets over prime fields}, Proceedings of the American Mathematical Society, \textbf{148}(8) (2020), 3209--3215.

\bibitem{Vi12} L. A. Vinh, \emph{Maximal sets of pairwise orthogonal vectors in finite fields}, Canad. Math. Bull. {\bf 55} (2012), no.2, 418--423.
%%%
\bibitem{37}
T. Wolff, \textit{Decay of circular means of Fourier transforms of measures}, Int. Math. Res. Not.
(1999), no. 10, 547--567.

\end{thebibliography}
\end{document}